\documentclass[preprint,12pt]{elsarticle}

\usepackage{graphicx}
\usepackage{amssymb, amsthm, amsmath, fullpage}
\usepackage{hyperref}
\usepackage{cleveref}
\usepackage[english]{babel}
\usepackage{bbm}
\usepackage{float}
\usepackage{comment}
\allowdisplaybreaks

\theoremstyle{plain}
\newtheorem{thm}{Theorem}[section]

\theoremstyle{definition}
\newtheorem{defn}[thm]{Definition}

\newtheorem{rem}[thm]{Remark}

\newtheorem{claim}[thm]{Claim}

\newtheorem{theorem}{Theorem}
\newtheorem{case}{Case}
\makeatletter
\@addtoreset{case}{subsection}

\newtheorem{subcase}{Subcase}[case]

\makeatother

\journal{Discrete Applied Mathematics}

\begin{document}

\begin{frontmatter}

\title{Off-diagonal Rado numbers for $x+y+c=z$ and $x+y+k=z$}

\author[1]{Rajat Adak}
\ead{rajatadak@iisc.ac.in}
\author[1]{Yash Bakshi\corref{cor1}}
\ead{yashbakshi@iisc.ac.in}
\author[1]{L. Sunil Chandran}
\ead{sunil@iisc.ac.in}
\author[1]{Saraswati Girish Nanoti}
\ead{saraswatig@iisc.ac.in}
\cortext[cor1]{Corresponding author}
\address[1]{Department of Computer Science and Automation, Indian Institute of Science, Bangalore, India}

\begin{abstract}
The study of Ramsey-type problems for linear equations originated with Schur's theorem 
and was later placed in a systematic framework by Richard Rado. In the 
off-diagonal setting, one fixes a pair of linear equations 
$(\mathcal{E}_1, \mathcal{E}_2)$ and seeks the least integer $N$ such that every 
red--blue coloring of $\{1,2,\dots,N\}$ contains either a red solution to 
$\mathcal{E}_1$ or a blue solution to $\mathcal{E}_2$. This threshold integer is referred to as the off-diagonal Rado number of the system $(\mathcal{E}_1, \mathcal{E}_2)$.
In this work, we study the discrete and continuous two-color off-diagonal Rado numbers for the nonhomogeneous linear equations $x+y+c=z$ and $x+y+k=z$, where $c\leq k$. In the discrete setting, $c$ and $k$ are nonnegative integers, whereas in the continuous setting, they are nonnegative real numbers. We determine the exact discrete and continuous two-color off-diagonal Rado numbers for this pair of shifted Schur equations.
\end{abstract}

\end{frontmatter}

\section{Introduction}

Schur's classical theorem \cite{Schur} establishes that every finite coloring of the positive integers $\mathbb {N} $ must contain a monochromatic solution to the equation $x+y=z$. Rado \cite{rado1933studien} substantially generalized Schur's theorem by characterizing the homogeneous systems of linear equations that are partition regular over the positive integers. Such a system is partition regular if every finite coloring of the positive integers admits a monochromatic solution. While the early literature largely concentrated on diagonal settings, where a single equation must have a monochromatic solution, attention has more recently shifted toward other variants, such as \emph{off-diagonal, disjunctive or simultaneous} versions.

In the off-diagonal setting, rather than seeking a monochromatic solution to a single equation, one seeks to guarantee that at least one of two equations admits a monochromatic solution in its prescribed color class. This viewpoint naturally extends classical Schur-type problems by allowing different equations to be associated with different colors. As a result, it leads to a richer framework that reveals new combinatorial behavior and provides a deeper structural understanding of these problems.

\begin{defn}[Rado number]
Let $\mathcal{E}$ be a linear equation with integer coefficients. 
The \emph{t-color Rado number} of $\mathcal{E}$, denoted by $R_t(\mathcal{E})$, is the smallest positive integer $N$ such that every $t$-coloring of the set $\{1,2,\dots,N\}$ contains a monochromatic solution to $\mathcal{E}$. If no such integer $N$ exists, then $R_t(\mathcal{E})$ is said to be infinite.
\end{defn}

\begin{defn}[Off-diagonal Rado number]
Let $\mathcal{E}_1$ and $\mathcal{E}_2$ be linear equations with integer coefficients. 
       The \emph{two-color off-diagonal Rado number} of the system of equations $(\mathcal{E}_1,\mathcal{E}_2)$, denoted by $R_2(\mathcal{E}_1,\mathcal{E}_2)$, is the smallest positive integer $N$ such that every red--blue coloring of the set $\{1,2,\dots,N\}$ contains either a red solution to $\mathcal{E}_1$ or a blue  solution to $\mathcal{E}_2$. If no such integer $N$ exists, then $R_2(\mathcal{E}_1,\mathcal{E}_2)$ is said to be infinite.
\end{defn} 

\begin{defn}[Off-diagonal continuous Rado number]
Let $\mathcal{E}_1$ and $\mathcal{E}_2$ be linear equations with real coefficients, and let $\alpha \in \mathbb{R}$. 
The \emph{two-color off-diagonal continuous Rado number} of the system of equations $(\mathcal{E}_1,\mathcal{E}_2)$ on $[\alpha,\infty)$, denoted by $R_{2}(\mathcal{E}_1,\mathcal{E}_2;\alpha)$, is the smallest real number $N \ge \alpha$ such that every red--blue coloring of the interval $[\alpha,N]$ contains either a red solution to $\mathcal{E}_1$ or a blue solution to $\mathcal{E}_2$. If no such real number $N$ exists, then $R_{2}(\mathcal{E}_1,\mathcal{E}_2;\alpha)$ is said to be infinite.
\end{defn}

\section{Previous Work}

The study of Rado numbers originates in classical results in Ramsey theory. Schur~\cite{Schur} proved that the equation $x+y=z$ is regular, introducing what are now known as Schur numbers. Rado~\cite{rado1933studien} generalized this work by characterizing all partition-regular linear systems, thereby identifying precisely which systems admit monochromatic solutions under every finite coloring.

Burr and Loo~\cite{BurrLoo1975} initiated the systematic study of exact Rado numbers for specific linear equations. In particular, they determined the two-color Rado number of the shifted Schur equation
\[
R_2(x+y+c=z)=4c+5
\]
for $c \ge -1$.  
 Beutelspacher and Brestovansky~\cite{beutelspacher2006generalized} proved that, for every positive integer $k$, the two-color Rado number of the generalized Schur equation is,
\[
R_2(x_1 + x_2 + \cdots + x_k = x_0)=k^2 + k - 1.
\]

A natural extension of this theory is the \emph{off-diagonal} setting, where instead of a single equation, one considers a pair of equations $(\mathcal{E}_1,\mathcal{E}_2)$ and seeks the smallest integer $N$ such that every red--blue coloring of $[1, N]$ yields either a red solution to $\mathcal{E}_1$ or a blue solution to $\mathcal{E}_2$. This framework generalizes classical Rado numbers by allowing different equations to be associated with different colors.

The first systematic study in this direction was carried out by Robertson and Schaal~\cite{robertson_schaal2001}, who introduced off-diagonal generalized Schur numbers. For positive integers $k$ and $l$, they determined the smallest integer $S(k,l)$ such that every red--blue coloring of $[1,S(k,l)]$ contains either a red solution to
\[
x_1 + x_2 + \cdots + x_{k-1} = x_k
\]
or a blue solution to
\[
x_1 + x_2 + \cdots + x_{l-1} = x_l.
\]

Further developments were made by Myers and Robertson~\cite{MyersRobertson2007}, who established general bounds for homogeneous linear equations ($x+qy=z, x+sy=z$) and obtained exact values in several special cases. Yao and Xia~\cite{yaoxia} subsequently obtained formulas for other homogeneous families, including
\[
R_2(3x+3y=z,3x+3qy=z)
\]
and
\[
R_2(2x+3y=z,2x+2qy=z).
\]
More recently, Jing and Mei~\cite{jingmie2024} determined
\[
R_2(2x+qy=2z, 2x+sy=2z)
\]
for odd integers \(q>s\geq 1\).

Sabo, Schaal and Tokaz \cite{sabo2007disjunctive} studied shifted Schur equations such as $ x_1+x_2+c=x_3$ and determined the \textit{disjunctive} Rado number. In 2025, Vestal and Glackin \cite{vestal2025linear} studied the same equations in the context of \textit{simultaneous} solutions. In this paper, we determine the
corresponding two-color \textit{off-diagonal} Rado number for the same equations 
\(x+y+c=z\) and \(x+y+k=z\), thereby covering another natural
Ramsey-type variant for this shifted Schur family. More recently, an \textit{off-diagonal} Rado number for a combination of homogeneous and nonhomogeneous equations has also been studied~\cite{adak2026off}. 

\section{Main Result}
\noindent
Throughout this paper, we denote the two-color off-diagonal Rado number by
\[
R_2(c,k):=R_2(x+y+c=z,\;x+y+k=z),
\]
where \(x+y+c=z\) is the red equation and \(x+y+k=z\) is the blue equation.
\subsection{\textbf{The discrete version}}
\begin{theorem} 
Let $c$ and $k$ be nonnegative integers with $c \le k$. Then the two-color off-diagonal Rado number $R_2(c,k)$ for the system 

\begin{align*}
x+y+c &= z \quad \text{(red)},\\
x+y+k  &= z \quad \text{(blue)},
\end{align*}
is given by
\[
R_2(c,k) =
\begin{cases}
\infty, & \text{if $c \not\equiv k\pmod{2}$} \\k+
3c+5,  & \text{if $c \equiv k\pmod{2}$ and $k \le2c$}\\
2k+c+4 & \text{if $c\equiv k\pmod{2}$ and $k>2c$ }
\end{cases}
\]
\end{theorem}

\begin{proof} We denote $x+y+c =z$ as the $c$-equation and $x+y+k=z$ as the $k$-equation. We evaluate $R_2(c,k)$ case by case.
\begin{case}
    $c$ and $k$ have opposite parity.
\end{case}
First, suppose $c$ is even, and $k$ is odd. Color every odd integer red and every even integer blue. 
If $x,y,z$ were a red solution to $x+y+c=z$, then $x,y,z$ would all be odd. However, the sum of two odd numbers is even, and as $c$ is also even, so $z$ would be even, a contradiction. 

Similarly, if $x,y,z$ were a blue solution to $x+y+k=z$, then $x,y,z$ would all be even. But the sum of two even numbers and an odd number is odd, so $z$ would be odd, again a contradiction. Hence, this coloring avoids both types of monochromatic solutions.

If $c$ is odd and $k$ is even, reverse the coloring: color every even integer red and every odd integer blue. In this case, any solution to $x+y+c=z$ would require $x,y,z$ to be even, but the sum of two even numbers and an odd number is odd; therefore, this is impossible. Likewise, any blue solution to $x+y+k=z$ would require $x,y,z$ to be odd, but the sum of two odd numbers and an even number is even, therefore this is again impossible.

Thus, for every $N$ there exists a coloring of $[1, N]$ with no red solution to $x+y+c=z$ and no blue solution to $x+y+k=z$. Therefore, the off-diagonal Rado number is infinite.

\begin{case}
    $c$ and $k$ have same parity with $k\le2c $.
\end{case}

\begin{claim}\label{clm1}
    $R_2(c,k) \geq k+3c+5$
\end{claim}
\begin{proof}
Let \[
\begin{aligned}
R &= \{1,2,\dots,c+1\} \cup \{k+2c+4,\dots,k+3c+4\}, \\
B &= \{c+2,c+3,\dots,k+2c+3\}.
\end{aligned}
\]
Consider the two-coloring of $[1,k+3c+4]$ defined by
\[
\chi(n)=
\begin{cases}
\text{Red}, & n \in R,\\
\text{Blue}, & n \in B.
\end{cases}
\]
Under this coloring, we will show that there is
\emph{no red solution} to $c$-equation  and \emph{no blue solution} to $k$-equation with all variables in $[1,k+3c+4]$.

\smallskip
\noindent\textbf{No blue solution to the $k$-equation:}\\
If $x$ and $y$ are blue, then $c+2\le x,y\le k+2c+3$ and, hence $z=x+y+k \ge (c+2)+(c+2)+k = k+2c+4$ which is red so, for any $x$, $y \in B$, $z$ is either red or outside the colored interval. So there is no blue solution to the $k$-equation.

\smallskip
\noindent\textbf{No red solution to the $c$-equation:}\\
If $x$ and $y$ are red from the first part of the interval, then $z$ will be blue as $z=x+y+c\ge 1+1+c=c+2$ and $z\le(c+1)+(c+1)+c=3c+2$ and since $c \le k$, $3c+2 < k+2c+3$.
If $x$ comes from the first part and $y$ from the second part of the Red interval, then $z \ge 1+(k+2c+4)+c=k+3c+5$, but $z$ should lie in the interval $[1,k+3c+4]$. Similarly, if both \(x\) and \(y\) come from the second red block, then $z=x+y+c\ge 2(k+2c+4)+c>k+3c+4$, so \(z\) lies outside the colored interval. Hence, there is no red solution to the $c$-equation.

\smallskip
\noindent
Consequently, this coloring avoids a red solution to the \(c\)-equation
and a blue solution to the \(k\)-equation. Therefore, the off-diagonal Rado number satisfies $R_2(c,k)\ge k+3c+5$.\end{proof}

\begin{claim}\label{clm2}
    $R_2(c,k) \leq k+3c+5$
\end{claim}
\begin{proof}We show that any red--blue coloring of set $[1,k+3c+5]$ gives either a red solution to $x+y+c=z$ or blue solution to $x+y+k=z$. Assume, for contradiction, there is a coloring such that no red solution to $x+y+c=z$ and no blue solution to $x+y+k=z$ exists, respectively. We start with an arbitrary coloring of numbers starting with $1$ (and using numbers up to $k+3c+5$), and try to extend this coloring such that the desired conditions are satisfied, but end up with either a red solution to the first equation or a blue solution to the second equation. 
\vspace{2mm}

\noindent\textbf{Notation.} We use the convention that, when trying to avoid monochromatic (red or blue) solutions, a given coloring with a particular solution is represented as an ordered triple \((x,y,z)\)--may force an extension of the coloring. For example if $1$ is red, taking $x = y =1$ in $c$-equation, to avoid a red solution, we get, 
$
(1,1,c+2)_c\;\Longrightarrow\; c+2 \text{ is blue}.
$
\begin{rem}
At boundary values, some forcing points may coincide. Each branch is
followed only until a contradiction occurs. A repeated same-color assignment
is redundant, whereas an opposite-color assignment terminates the branch.
This includes \(c=k=0\).
\end{rem}
\begin{subcase} Assume that $1$ is colored red. \end{subcase}
\begingroup
\begin{align*}
(1,1,c+2)_c&\Longrightarrow {c+2} \text{ is blue}.\\
(c+2,c+2,k+2c+4)_k&\Longrightarrow {k+2c+4} \text{ is red}.\\
(\tfrac{k+c}{2}+2,\tfrac{k+c}{2}+2,k+2c+4)_c&\Longrightarrow {\tfrac{k+c}{2}+2} \text{ is blue}.\\
\left(\tfrac{k+c}{2}+2,\tfrac{k+c}{2}+2,2k+c+4\right)_k&\Longrightarrow {2k+c+4} \text{ is red}.\\
(1,k+2c+4,k+3c+5)_c&\Longrightarrow {k+3c+5} \text{ is blue}.\\
(2c+3,c+2,k+3c+5)_k&\Longrightarrow {2c+3} \text{ is red}.\\
(1,2c+3,3c+4)_c&\Longrightarrow {3c+4} \text{ is blue}.\\
(2c-k+2,c+2,3c+4)_k&\Longrightarrow {2c-k+2} \text{ is red}.\\
(1,2c-k+2,3c-k+3)_c&\Longrightarrow {3c-k+3} \text{ is blue}.\\
(k+2,k+2,2k+c+4)_c&\Longrightarrow {k+2} \text{ is blue}.\\
\end{align*}
Thus, $(k+2,3c-k+3,k+3c+5)$ is a blue solution to the $k$-equation, a contradiction.
\endgroup

\begin{subcase}
    $1$ is colored blue.
\end{subcase} 
\begingroup

\begin{align*}
(1,1,k+2)_k&\Longrightarrow {k+2} \text{ is red}.\\
(k+2,k+2,2k+c+4)_c&\Longrightarrow {2k+c+4} \text{ is blue}.\\
\left(\tfrac{k+c}{2}+2,\tfrac{k+c}{2}+2,2k+c+4\right)_k&\Longrightarrow {\tfrac{k+c}{2}+2} \text{ is red}.\\
\left(\tfrac{k+c}{2}+2,\tfrac{k+c}{2}+2,k+2c+4\right)_c&\Longrightarrow {k+2c+4} \text{ is blue}.\\
(c+2,c+2,k+2c+4)_k&\Longrightarrow {c+2} \text{ is red}.\\
(c+2,c+2,3c+4)_c&\Longrightarrow {3c+4} \text{ is blue}.\\
(1,3c+4,k+3c+5)_k&\Longrightarrow {k+3c+5} \text{ is red}.\\
(k+c+3,c+2,k+3c+5)_c&\Longrightarrow {k+c+3} \text{ is blue}.\\
\end{align*}
Thus $(1,k+c+3,2k+c+4)_k$ is a blue solution to the $k$-equation, a contradiction.
\endgroup
Thus, we arrive at a contradiction in both the subcases. Therefore, $R_2(c,k) \leq k+3c+5$.
\end{proof}
From Claim \ref{clm1} and \ref{clm2} we get that when $c\le k\le2c$, $R_2(c,k) = k+3c+5$.

\begin{rem}
Note that $k+3c+5 > 2k+c+4$ as $k \le 2c$ and $2c-k+1 >0$. Also, $\frac{k+c}{2}$ is an integer because $c$ and $k$ have the same parity.
\end{rem}
\begin{case}
     $c$ and $k$ have the same parity with $k>2c $.
\end{case}
\begin{claim}\label{clm3.3}
    $R_2(c,k) \geq 2k+c+4$.
\end{claim} 

\begin{proof}
Let \[
\begin{aligned}
R &= \{k+2,\dots,2k+c+3\}, \\
B &= \{1,2,\dots,k+1\}.
\end{aligned}
\]
 Consider the $2$-coloring of $[1,2k+c+3]$ defined by,
\[
\chi(n)=
\begin{cases}
\text{Red}, & n \in R,\\
\text{Blue}, & n \in B.
\end{cases}
\]
We show that under this coloring, there is neither a red solution to the $c$-equation, nor a blue solution to the $k$-equation.

\smallskip
\noindent\textbf{No blue solution to the $k$-equation:}\\
If both $x$ and $y \in B$, then
$
z=x+y+k\ge 1+1+k=k+2.
$
Thus, if \(z\le 2k+c+3\), then \(z\in R\), whereas if
\(z>2k+c+3\), then \(z\) lies outside the colored interval. Hence,
there is no blue solution to the \(k\)-equation.

\smallskip
\noindent\textbf{No red solution to the $c$-equation:}\\
If both $x$ and $y \in R$, then $z=x+y+c \ge (k+2)+(k+2)+c=2k+c+4$ which is outside the colored interval. Therefore, no red solution to the $c$-equation exists. \\
We get $R_2(c,k) > 2k + c + 3$ as a lower bound. Hence, $R_2(c,k) \ge 2k+c+4$.
\end{proof}
\begin{claim}\label{clm3.4}
    $R_2(c,k) \leq 2k+c+4$
\end{claim}
\begin{proof} Suppose, for a contradiction, that there exists a red--blue coloring of $\{1,2,\dots,2k+c+4\}$ containing neither a red solution to the $c$-equation nor a blue solution to the $k$-equation. We derive a
contradiction by analyzing the color assigned to $1$.

\begin{subcase}
     $1$ is colored red.
\end{subcase}
\begingroup
\begin{align*}
(1,1,c+2)_c&\Longrightarrow {c+2} \text{ is blue}.\\
(c+2,c+2,k+2c+4)_k&\Longrightarrow {k+2c+4}\text{ is red}.\\
(\frac{k+c}{2}+2,\frac{k+c}{2}+2,k+2c+4)_c&\Longrightarrow {\frac{k+c}{2}+2}\text{ is blue}.
\end{align*}
Since \(k>2c\), we have \(k-2c\geq1\). Hence \(k-2c\in[1,2k+c+4]\), and it is valid to consider its color.

Suppose $k-2c$ is red,
\begin{align*}
(1,k-2c,k-c+1)_c
  &\Longrightarrow {k-c+1}\text{ is blue}.\\
  (k-c+1,c+2,2k+3)_k
  &\Longrightarrow {2k+3}\text{ is red}.\\
  (1,2k+3,2k+c+4)_c
  &\Longrightarrow {2k+c+4}\text{ is blue}.
\end{align*}
Thus $(\frac{k+c}{2}+2,\frac{k+c}{2}+2,2k+c+4)_k$ is a blue solution, a contradiction.

Suppose $k-2c$ is blue,
\begin{align*}
(c+2,k-2c,2k-c+2)_k
  &\Longrightarrow {2k-c+2}\text{ is red}.\\
  (k-c+1,k-c+1,2k-c+2)_c
  &\Longrightarrow {k-c+1}\text{ is blue}.\\
  (k-c+1,c+2,2k+3)_k
  &\Longrightarrow {2k+3}\text{ is red}.\\
    (1,2k+3,2k+c+4)_c
  &\Longrightarrow {2k+c+4}\text{ is blue}.
\end{align*}
Thus $(\frac{k+c}{2}+2,\frac{k+c}{2}+2,2k+c+4)_k$ is a blue solution, a contradiction.
\endgroup

\begin{subcase}
     $1$ is colored blue.
\end{subcase}
\begingroup
\begin{align*}
(1,1,k+2)_k&\Longrightarrow k+2 \text{ is red.}\\
(k+2,k+2,2k+c+4)_c&\Longrightarrow 2k+c+4 \text{ is blue.}
\end{align*}
Suppose $k-2c$ is red,
\begin{align*}
(k-2c,k+2,2k-c+2)_c&\Longrightarrow 2k-c+2 \text{ is blue.}\\
(\frac{k-c}{2}+1,\frac{k-c}{2}+1,2k-c+2)_k&\Longrightarrow \frac{k-c}{2}+1 \text{ is red.}
\end{align*}
Thus $(\frac{k-c}{2}+1,\frac{k-c}{2}+1,k+2)_c$\text{ is a red solution, a contradiction.}\\
Suppose $k-2c$ is blue,
\begin{align*}
(\frac{k+c}{2}+2,\frac{k+c}{2}+2,2k+c+4)_k&\Longrightarrow \frac{k+c}{2}+2 \text{ is red.}\\
(k-2c,3c+4,2k+c+4)_k&\Longrightarrow 3c+4 \text{ is red.}\\
(c+2,c+2,3c+4)_c&\Longrightarrow c+2 \text{ is blue.}\\
(c+2,c+2,k+2c+4)_k&\Longrightarrow k+2c+4 \text{ is red.}
\end{align*}
Thus $(\frac{k+c}{2}+2,\frac{k+c}{2}+2,k+2c+4)_c$\text{ is a red solution, a contradiction.}
\endgroup
Thus, we arrive at a contradiction in both the subcases. Therefore, $R_2(c,k) \leq 2k+c+4$. \end{proof}
From Claim \ref{clm3.3} and \ref{clm3.4} we get that when $k>2c$, $R_2(c,k) = 2k+c+4$.\\

Thus, we get the two-color off-diagonal Rado number for all three cases. Now we extend this to a two-color continuous off-diagonal Rado number. 
\end{proof}

\subsection{\textbf{The continuous version}}

\begin{theorem} 
Let $c$ and $k$ be nonnegative real numbers with $c \le k$ and $\alpha\ge0$. Then the two-color off-diagonal continuous Rado number $R_2(c,k ;\alpha)$ on the interval $[\alpha, N]$ for the system 
\begin{align*}
x+y+c &= z \quad \text{(red)},\\
x+y+k  &= z \quad \text{(blue)},
\end{align*}
is given by
\[
R_2(c,k ; \alpha) =
\begin{cases}
k+3c+5\alpha,  & \text{if $k \le2c+\alpha$} \\
2k+c+4\alpha & \text{if $k>2c+\alpha$ }
\end{cases}
\]
\end{theorem}

\begin{proof}

 \begin{case}
    When $k\le2c+\alpha$.
\end{case}
\begin{claim}\label{clm4.2}
    $R_2(c,k;\alpha) = k+3c+5\alpha$
\end{claim}
\begin{proof}The lower bound is given by the following coloring.\\
\[
\begin{aligned}
R &= [\alpha,c+2\alpha) \cup [k+2c+4\alpha,k+3c+5\alpha), \\
B &= [c+2\alpha,k+2c+4\alpha).
\end{aligned}
\]
This coloring of the interval $[\alpha,k+3c+5\alpha)$ contains neither a red solution to the $c$-equation nor a blue solution to the $k$-equation. The verification is analogous to that of the corresponding discrete lower-bound construction. Hence, for every $N<k+3c+5\alpha$, its restriction to $[\alpha,N]$ avoids both forbidden monochromatic solutions. Therefore,
\[
R_2(c,k;\alpha)\geq k+3c+5\alpha.
\]
For the upper bound, we again consider coloring the interval $[\alpha, k+3c+5\alpha]$.

\begin{subcase} Assume that $\alpha$ is colored red. \end{subcase}
\begingroup
\begin{align*}
(\alpha,\alpha,c+2\alpha)_c&\Longrightarrow {c+2\alpha} \text{ is blue}.\\
(c+2\alpha,c+2\alpha,k+2c+4\alpha)_k&\Longrightarrow {k+2c+4\alpha} \text{ is red}.\\
(\tfrac{k+c}{2}+2\alpha,\tfrac{k+c}{2}+2\alpha,k+2c+4\alpha)_c&\Longrightarrow {\tfrac{k+c}{2}+2\alpha} \text{ is blue}.\\
\left(\tfrac{k+c}{2}+2\alpha,\tfrac{k+c}{2}+2\alpha,2k+c+4\alpha \right)_k&\Longrightarrow {2k+c+4\alpha} \text{ is red}.\\
(\alpha,k+2c+4\alpha,k+3c+5\alpha)_c&\Longrightarrow {k+3c+5\alpha} \text{ is blue}.\\
(2c+3\alpha,c+2\alpha,k+3c+5\alpha)_k&\Longrightarrow {2c+3\alpha} \text{ is red}.\\
(\alpha,2c+3\alpha,3c+4\alpha)_c&\Longrightarrow {3c+4\alpha} \text{ is blue}.\\
(2c-k+2\alpha,c+2\alpha,3c+4\alpha)_k&\Longrightarrow {2c-k+2\alpha} \text{ is red}.\\
(\alpha,2c-k+2\alpha,3c-k+3\alpha)_c&\Longrightarrow {3c-k+3\alpha} \text{ is blue}.\\
(k+2\alpha,k+2\alpha,2k+c+4\alpha)_c&\Longrightarrow {k+2\alpha} \text{ is blue}.\\
\end{align*}
Thus, \((k+2\alpha,3c-k+3\alpha,k+3c+5\alpha)\) is a blue solution to the
\(k\)-equation, a contradiction.
\endgroup
\begin{subcase}
    $\alpha$ is colored blue.
\end{subcase} 
\begingroup

\begin{align*}
(\alpha,\alpha,k+2\alpha)_k&\Longrightarrow {k+2\alpha} \text{ is red}.\\
(k+2\alpha,k+2\alpha,2k+c+4\alpha)_c&\Longrightarrow {2k+c+4\alpha} \text{ is blue}.\\
\left(\tfrac{k+c}{2}+2\alpha,\tfrac{k+c}{2}+2\alpha,2k+c+4\alpha\right)_k&\Longrightarrow {\tfrac{k+c}{2}+2\alpha} \text{ is red}.\\
\left(\tfrac{k+c}{2}+2\alpha,\tfrac{k+c}{2}+2\alpha,k+2c+4\alpha\right)_c&\Longrightarrow {k+2c+4\alpha} \text{ is blue}.\\
(c+2\alpha,c+2\alpha,k+2c+4\alpha)_k&\Longrightarrow {c+2\alpha} \text{ is red}.\\
(c+2\alpha,c+2\alpha,3c+4\alpha)_c&\Longrightarrow {3c+4\alpha} \text{ is blue}.\\
(\alpha,3c+4\alpha,k+3c+5\alpha)_k&\Longrightarrow {k+3c+5\alpha} \text{ is red}.\\
(k+c+3\alpha,c+2\alpha,k+3c+5\alpha)_c&\Longrightarrow {k+c+3\alpha} \text{ is blue}.\\
\end{align*}
Thus $(\alpha,k+c+3\alpha,2k+c+4\alpha)$ is a blue solution to the $k$-equation, a contradiction.
\endgroup
Thus, we arrive at a contradiction in both the subcases. Therefore, $R_2(c,k;\alpha) \leq k+3c+5\alpha$.
\end{proof}
\begin{case}
    When $k>2c+\alpha$.
\end{case}
\begin{claim}\label{clm3.1}
    $R_2(c,k;\alpha) = 2k+c+4\alpha$
\end{claim}
\begin{proof}The lower bound is given by the following coloring.\\
\[
\begin{aligned}
R &= [k+2\alpha,2k+c+4\alpha), \\
B &= [\alpha,k+2\alpha).
\end{aligned}
\]
The above coloring of the interval $[\alpha,2k+c+4\alpha)$ avoids both the red solution to the $c$-equation and the blue solution to the $k$-equation.
Therefore, $R_2(c,k;\alpha) \ge  2k+c+4\alpha$\\
For the upper bound, again we can consider coloring the interval $[\alpha, 2k+c+4\alpha]$.

\begin{subcase}
     $\alpha$ is colored red.
\end{subcase}
\begingroup
\begin{align*}
(\alpha,\alpha,c+2\alpha)_c&\Longrightarrow {c+2\alpha} \text{ is blue}.\\
(c+2\alpha,c+2\alpha,k+2c+4\alpha)_k&\Longrightarrow {k+2c+4\alpha}\text{ is red}.\\
(\frac{k+c}{2}+2\alpha,\frac{k+c}{2}+2\alpha,k+2c+4\alpha)_c&\Longrightarrow {\frac{k+c}{2}+2\alpha}\text{ is blue}.
\end{align*}
Since \(k>2c+\alpha\), we have \(k-2c>\alpha\). Hence the point
\(k-2c\) lies in the interval \([\alpha,2k+c+4\alpha]\), and it is
valid to consider its color.\\
   Suppose $k-2c$ is red,
\begin{align*}
(\alpha,k-2c,k-c+\alpha)_c
  &\Longrightarrow {k-c+\alpha}\text{ is blue}.\\
  (k-c+\alpha,c+2\alpha,2k+3\alpha)_k
  &\Longrightarrow {2k+3\alpha}\text{ is red}.\\
  (\alpha,2k+3\alpha,2k+c+4\alpha)_c
  &\Longrightarrow {2k+c+4\alpha}\text{ is blue}.
\end{align*}
Thus $(\frac{k+c}{2}+2\alpha,\frac{k+c}{2}+2\alpha,2k+c+4\alpha)_k$ is a blue solution, a contradiction.\\
Suppose $k-2c$ is blue,
\begin{align*}
(c+2\alpha,k-2c,2k-c+2\alpha)_k
  &\Longrightarrow {2k-c+2\alpha}\text{ is red}.\\
  (k-c+\alpha,k-c+\alpha,2k-c+2\alpha)_c
  &\Longrightarrow {k-c+\alpha}\text{ is blue}.\\
  (k-c+\alpha,c+2\alpha,2k+3\alpha)_k
  &\Longrightarrow {2k+3\alpha}\text{ is red}.\\
    (\alpha,2k+3\alpha,2k+c+4\alpha)_c
  &\Longrightarrow {2k+c+4\alpha}\text{ is blue}.
\end{align*}
Thus $(\frac{k+c}{2}+2\alpha,\frac{k+c}{2}+2\alpha,2k+c+4\alpha)_k$ is a blue solution, a contradiction.
\endgroup

\begin{subcase}
     $\alpha$ is colored blue.
\end{subcase}
\begingroup
\begin{align*}
(\alpha,\alpha,k+2\alpha)_k&\Longrightarrow k+2\alpha \text{ is red.}\\
(k+2\alpha,k+2\alpha,2k+c+4\alpha)_c&\Longrightarrow 2k+c+4\alpha \text{ is blue.}
\end{align*}
Suppose $k-2c$ is red,
\begin{align*}
(k-2c,k+2\alpha,2k-c+2\alpha)_c&\Longrightarrow 2k-c+2\alpha \text{ is blue.}\\
(\frac{k-c}{2}+\alpha,\frac{k-c}{2}+\alpha,2k-c+2\alpha)_k&\Longrightarrow \frac{k-c}{2}+\alpha \text{ is red.}
\end{align*}
Thus $(\frac{k-c}{2}+\alpha,\frac{k-c}{2}+\alpha,k+2\alpha)_c$\text{ is a red solution, a contradiction.}\\

Suppose $k-2c$ is blue,
\begin{align*}
(\frac{k+c}{2}+2\alpha,\frac{k+c}{2}+2\alpha,2k+c+4\alpha)_k&\Longrightarrow \frac{k+c}{2}+2\alpha \text{ is red.}\\
(k-2c,3c+4\alpha,2k+c+4\alpha)_k&\Longrightarrow 3c+4\alpha \text{ is red.}\\
(c+2\alpha,c+2\alpha,3c+4\alpha)_c&\Longrightarrow c+2\alpha \text{ is blue.}\\
(c+2\alpha,c+2\alpha,k+2c+4\alpha)_k&\Longrightarrow k+2c+4\alpha \text{ is red.}
\end{align*}
Thus $(\frac{k+c}{2}+2\alpha,\frac{k+c}{2}+2\alpha,k+2c+4\alpha)_c$\text{ is a red solution, a contradiction.}
\endgroup
Thus, we arrive at a contradiction in both the subcases. Therefore, $R_2(c,k;\alpha) \leq 2k+c+4\alpha$.
\end{proof}

\begin{rem}
The parity condition appearing in the discrete version has no analog
in the continuous version. Its role in the proof of the discrete version is to
ensure that certain midpoint expressions are integers. Since the
continuous problem allows \(x,y,z\) to range over real intervals, these
midpoints are admissible real numbers, and the continuous formula
depends only on whether \(k\leq 2c+\alpha\) or \(k>2c+\alpha\).
\end{rem}
\begin{rem}
For the shifted Schur pair \(x+y+c=z\) and \(x+y+k=z\), with nonnegative real numbers \( c\le k\) and \(\alpha\ge0\), the continuous two-color off-diagonal
Rado number is always \textit{finite}. Unlike the discrete
case, where opposite parity between $c$ and $k$ can lead to an infinite number by admitting a valid two-coloring of the solution set with no
monochromatic solution.
\end{rem}
\end{proof}
\section{Conclusion}
The theory of Rado numbers involves numerous open problems, whose complexity increases significantly with the number of colors. As a result, much of the literature has focused on the two-color case. In this setting, several variants—including disjunctive~\cite{johnson2005disjunctive} and off-diagonal~\cite{ jingmie2024, MyersRobertson2007,robertson_schaal2001} and continuous~\cite{brady2005rado,vestal2025off,vestal2025linear} Rado numbers—have been studied for various classes of linear equations and systems. Nevertheless, the existing work on off-diagonal Rado numbers is largely limited to homogeneous equations, with the nonhomogeneous case remaining comparatively underexplored.

In this work, we determined the $2$-color off-diagonal Rado numbers $R_2(c,k)$ for $x+y+c=z$ and  $x+y+k=z$ for both discrete and continuous versions. There are several directions open for future investigation, including extending the 
present framework to broader families of equations encompassing both homogeneous 
and nonhomogeneous cases, as well as generalizing the known results to colorings 
with more than two colors.
\bibliographystyle{plain}
\bibliography{references}

@article{Schur,
  author  = {Schur, Issai},
  title   = {{\"U}ber die Kongruenz $x^m+y^m\equiv z^m\pmod p$},
  journal = {Jahresbericht der Deutschen Mathematiker-Vereinigung},
  volume  = {25},
  pages   = {114--117},
  year    = {1916}
}

@article{BurrLoo1975,
  author  = {Burr, Stefan A. and Loo, S. J.},
  title   = {On the {R}ado Numbers of Some Linear Equations},
  journal = {Preprint},
  year    = {1975}
}

@article{robertson_schaal2001,
  author  = {Robertson, Aaron and Schaal, Daniel},
  title   = {Off-Diagonal Generalized {S}chur Numbers},
  journal = {Advances in Applied Mathematics},
  volume  = {26},
  number  = {3},
  pages   = {252--257},
  year    = {2001},
  doi     = {10.1006/aama.2000.0718},
}

@article{johnson2005disjunctive,
  author  = {Johnson, Brenda and Schaal, Daniel},
  title   = {Disjunctive {R}ado Numbers},
  journal = {Journal of Combinatorial Theory, Series A},
  volume  = {112},
  number  = {2},
  pages   = {263--276},
  year    = {2005},
  doi     = {10.1016/j.jcta.2005.02.007},
}

@article{MyersRobertson2007,
  author  = {Myers, Kellen and Robertson, Aaron},
  title   = {Two-Color Off-Diagonal {R}ado-Type Numbers},
  journal = {The Electronic Journal of Combinatorics},
  volume  = {14},
  number  = {1},
  pages   = {R53},
  year    = {2007},
  doi     = {10.37236/971},
}

@inproceedings{beutelspacher2006generalized,
  author    = {Beutelspacher, Albrecht and Brestovansky, Walter},
  title     = {Generalized {S}chur Numbers},
  booktitle = {Combinatorial Theory},
  editor    = {Jungnickel, Dieter and Vedder, Klaus},
  series    = {Lecture Notes in Mathematics},
  volume    = {969},
  pages     = {30--38},
  publisher = {Springer},
  address   = {Berlin, Heidelberg},
  year      = {1982},
  doi       = {10.1007/BFb0062985},
}

@article{rado1933studien,
  author  = {Rado, Richard},
  title   = {Studien zur Kombinatorik},
  journal = {Mathematische Zeitschrift},
  volume  = {36},
  pages   = {424--480},
  year    = {1933},
}

@article{sabo2007disjunctive,
  author  = {Sabo, Dusty and Schaal, Daniel and Tokaz, Jacent},
  title   = {Disjunctive {R}ado Numbers for $x_1+x_2+c=x_3$},
  journal = {Integers},
  volume  = {7},
  number  = {1},
  pages   = {A29},
  year    = {2007},
}

@article{vestal2025off,
  author  = {Vestal, Don and Sax, Jonathan},
  title   = {Off-Diagonal Continuous {R}ado Numbers $x_1+x_2+\dots+x_k=x_0$},
  journal = {arXiv preprint arXiv:2511.20528},
  year    = {2025},
  eprint  = {2511.20528},
  archiveprefix = {arXiv},
}

@article{vestal2025linear,
  author  = {Vestal, Donald L., Jr. and Glackin, Anthony},
  title   = {A Linear System Involving the Equation $x_1+x_2+c=x_0$},
  journal = {Congressus Numerantium},
  volume  = {236},
  year    = {2025},
}

@article{yaoxia,
  author  = {Yao, Olivia X. M. and Xia, Ernest X. W.},
  title   = {Two Formulas of 2-Color Off-Diagonal {R}ado Numbers},
  journal = {Graphs and Combinatorics},
  volume  = {31},
  number  = {1},
  pages   = {299--307},
  year    = {2015},
  doi     = {10.1007/s00373-013-1378-9},
}

@article{adak2026off,
  author  = {Adak, Rajat and Bakshi, Yash and Chandran, L. Sunil and Nanoti, Saraswati Girish},
  title   = {Off-Diagonal {R}ado Number for $x+y+c=z$ and $x+qy=z$},
  journal = {arXiv preprint arXiv:2602.23954},
  year    = {2026},
  eprint  = {2602.23954},
  archiveprefix = {arXiv},
}

@article{jingmie2024,
  author  = {Jing, Jin and Mei, Y. M.},
  title   = {On Some Exact Formulas for 2-Color Off-Diagonal {R}ado Numbers},
  journal = {Journal of Combinatorial Mathematics and Combinatorial Computing},
  volume  = {119},
  pages   = {75--83},
  year    = {2024},
  doi     = {10.61091/jcmcc119-08},
}

@article{brady2005rado,
  author  = {Brady, Caitlin S. and Haas, Ruth},
  title   = {{R}ado Numbers for the Real Line},
  journal = {Congressus Numerantium},
  volume  = {177},
  pages   = {109--114},
  year    = {2005}
}
\end{document}